\documentclass{zzz} 
\usepackage{graphics,epsfig,color}

\newcommand{\N}{{\mathbf N}}

\newcommand{\R}{{\mathbf R}}

\newcommand{\Z}{{\mathbf Z}}
\newcommand{\C}{{\mathbf C}}
\newcommand{\RR}{{\rm R}}

\newtheorem{thm}[lemma]{Theorem}
\newtheorem{cor}[lemma]{Corollary}
\newtheorem{defn}[lemma]{Definition}

\makeatletter
\def\eqnarray{\stepcounter{equation}\let\@currentlabel=\theequation
\global\@eqnswtrue
\tabskip\@centering\let\\=\@eqncr
$$\halign to \displaywidth\bgroup\hfil\global\@eqcnt\z@
  $\displaystyle\tabskip\z@{##}$&\global\@eqcnt\@ne
  \hfil$\displaystyle{{}##{}}$\hfil
  &\global\@eqcnt\tw@ $\displaystyle{##}$\hfil
  \tabskip\@centering&\llap{##}\tabskip\z@\cr}

\def\endeqnarray{\@@eqncr\egroup
      \global\advance\c@equation\m@ne$$\global\@ignoretrue}

\def\@yeqncr{\@ifnextchar [{\@xeqncr}{\@xeqncr[5pt]}}
\makeatother

\begin{document}

\renewcommand{\PaperNumber}{***}

\FirstPageHeading

\ShortArticleName{
Generalizations of
generating function 
for
hypergeometric orthogonal polynomials}

\ArticleName{Generalizations of 
generating functions for hypergeometric orthogonal polynomials with definite integrals}

\Author{Howard S.~COHL, $^\dag\!\!\ $ Connor MACKENZIE $^\S$$^\dag$
and Hans VOLKMER $^\ddag$} 
\AuthorNameForHeading{H.~S.~Cohl, C. MacKenzie \& H.~Volkmer}

\Address{$^\dag$~Applied and Computational Mathematics Division, 
National Institute of Standards and Technology, 
Gaithersburg, MD 20899-8910, USA
} 
\EmailD{howard.cohl@nist.gov} 

\URLaddressD{http://hcohl.sdf.org} 

\Address{$^\S$~Department of Mathematics, Westminster College, 319 South Market Street, 
New Wilmington, PA 16172, USA
}
\EmailD{mackcm22@wclive.westminster.edu} 

\Address{$^\ddag$~Department of Mathematical Sciences, 
University of Wisconsin-Milwaukee, P.O. Box 413, 
Milwaukee, WI  53201, USA
}
\EmailD{volkmer@uwm.edu} 


\ArticleDates{Received XX September 2012 in final form ????; Published online ????}

\Abstract{We generalize generating functions for hypergeometric orthogonal
polynomials, namely 
Jacobi, Gegenbauer,
Laguerre, and 
Wilson 
polynomials.  
These generalizations of generating functions are accomplished through 
series rearrangement using connection relations with one free parameter
for these orthogonal polynomials.  We also use orthogonality relations 
to determine corresponding definite integrals.}


\Classification{33C45, 05A15, 33C20, 34L10, 30E20}

\section{Introduction}
\label{Introduction}

In this paper, we apply connection relations 
(see for instance Andrews {\it et al.}~(1999) \cite[Section 7.1]{AAR};
Askey (1975) \cite[Lecture 7]{Askey75})
with one free parameter 
for the Jacobi, Gegenbauer, Laguerre, and Wilson polynomials 
(see Chapter 18 in \cite{NIST})
to generalize generating functions for these orthogonal polynomials
using series rearrangement. 
This is because connection relations with one free parameter only involve a 
summation 
over products of gamma functions and are straightforward to sum.
We have already applied our series rearrangment technique using a connection relation
with one free parameter to the generating 
function for Gegenbauer polynomials
\cite[(18.12.4)]{NIST}
\begin{equation}
\frac{1}{(1+\rho^2-2\rho x)^\nu}=\sum_{n=0}^\infty\rho^nC_n^\nu(x).
\label{generatingfnforgeg}
\end{equation}
The connection relation for Gegenbauer polynomials is
given in Olver {\it et al.}~(2010) \cite[(18.18.16)]{NIST} 
(see also Ismail (2005) \cite[(9.1.2)]{Ismail}), namely
\begin{equation}
C_n^\nu(x)=\frac{1}{\mu}\sum_{k=0}^{\lfloor n/2\rfloor} 
(\mu+n-2k)
\frac{
(\nu-\mu)_k\,(\nu)_{n-k}}{k!(\mu+1)_{n-k}}
C_{n-2k}^\mu(x).
\label{GegenConnect}
\end{equation}
Inserting (\ref{GegenConnect}) into (\ref{generatingfnforgeg}), we 
obtained a result 
\cite[(1)]{CohlGenGegen}
which generalizes 
(\ref{generatingfnforgeg}), namely
\begin{equation}
\frac{1}{(1+\rho^2-2\rho x)^\nu}=\sum_{n=0}^\infty f_n^{(\nu,\mu)}(\rho)C_n^\mu(x),
\label{generatingfnforgeggen}
\end{equation}
where $f_n^{(\nu,\mu)}:\{z\in\C:0<|z|<1\}\setminus(-1,0)\to\C$ is defined by
\[
f_n^{(\nu,\mu)}(\rho):=\frac{\Gamma(\mu)e^{i\pi(\mu-\nu+1/2)}(n+\mu)}
{\sqrt{\pi}\,\Gamma(\nu)\rho^{\mu+1/2}(1-\rho^2)^{\nu-\mu-1/2}}
Q_{n+\mu-1/2}^{\nu-\mu-1/2}\left(\frac{1+\rho^2}{2\rho}\right),
\]
where $Q_\nu^\mu$ is the associated Legendre function of the second kind
\cite[Chapter 14]{NIST}.  It is easy to demonstrate that 
$f_n^{(\nu,\nu)}(\rho)=\rho^n$. We have also succesfully applied this technique to 
an extension of (\ref{generatingfnforgeg}) expanded in Jacobi polynomials
using a connection relation with two free parameters in Cohl (2010) 
\cite[Theorem 5.1]{Cohl12pow}.  In this case the coefficients of the expansion are
given in terms of Jacobi functions of the second kind.
Applying this technique 
with connection relations with more than one free parameter is therefore
possible, but it is more intricate and involves rearrangement and 
summation of three or more increasingly complicated sums.  The goal of
this paper is to demonstrate the effectiveness of the series rearrangement 
technique using connection relations with one free parameter by applying
it to some of the most fundamental generating functions for hypergeometric
orthogonal polynomials.

\medskip

\noindent Unless otherwise stated, the domains of convergence given
in this paper are those of the original generating function and/or
its corresponding definite integral.  In this paper, we only justify
summation interchanges for a few of the theorems we present.
For the interchange justifications we have given, we give all the details.
However, for the sake of brevity, we leave justification for the 
remaining interchanges to the reader.

\medskip

\noindent Here we will make a short review of the special functions which
are used in this paper.  The generalized hypergeometric function
${_p}F_q:\C^p\times(\C\setminus\N_0)^q\times\left\{z\in\C:|z|<1\right\}\to\C$ 
(see Chapter 16 in Olver {\it et al.}~(2010) \cite{NIST}) is defined as
\[
{_p}F_q\left(
\begin{array}{c}
a_1,\dots,a_p\\[0.2cm]
b_1,\dots,b_q
\end{array};z\right)
:=\sum_{n=0}^\infty \frac{(a_1)_n\,\dots\,(a_p)_n}{(b_1)_n\,\dots\,(b_q)_n} \frac{z^n}{n!}.
\]
When $p=2,\,q=1,$ this is the special case referred to as the Gauss hypergeometric function
${}_2F_1:\C^2\times(\C\setminus -\N_0)\times\left\{z\in\C:|z|<1\right\}\to\C$ (see Chapter
15 of Olver {\it et al.}~(2010) \cite{NIST}).
When $p=1,\,q=1$ this is Kummer's confluent hypergeometric function of the first kind
$M:\C\times(\C\setminus\N_0)\times\C\to\C$ (see Chapter 13 in Olver {\it et al.}~(2010)
\cite{NIST}), namely
\medskip
\[
M(a,b,z):=\sum_{n=0}^\infty\frac{(a)_n}{(b)_n}\frac{z^n}{n!}=
{_1}F_1\left(\begin{array}{c}a\\b\end{array};z\right).
\]
When $p=0$, $q=1$, this is related to the Bessel function of the first kind
(see Chapter 10 in Olver {\it et al.}~(2010) \cite{NIST})
$J_\nu:\C\setminus\{0\}\to\C$, for $\nu\in\C$, defined by
\begin{equation}
J_\nu(z):=\frac{(z/2)^\nu}{\Gamma(\nu+1)}\,{}_0F_1\left(
\begin{array}{c}-\\ \nu+1
\end{array};\frac{-z^2}{4}
\right).
\label{BesselJ}
\end{equation}
The case of $p=0$, $q=1$ is also related to the 
modified Bessel function of the first kind
(see Chapter 10 in Olver {\it et al.}~(2010) \cite{NIST})
$I_\nu:\C\setminus(-\infty,0]\to\C$, for $\nu\in\C$, defined by
\[
I_\nu(z):=\frac{(z/2)^\nu}{\Gamma(\nu+1)}\,{_0}F_1\left(
\begin{array}{c}-\\ \nu+1
\end{array};\frac{z^2}{4}\right).
\]
When $p=1$, $q=0$, this is the binomial expansion
(see for instance Olver {\it et al.}~(2010) \cite[(15.4.6)]{NIST}), namely
\begin{equation}
{_1}F_0\left(
\begin{array}{c}
\alpha\\[0.2cm]
-
\end{array};z\right)=(1-z)^{-\alpha}.
\label{binomialexp}
\end{equation}

\medskip

\noindent In these sums, the Pochhammer symbol (rising factorial) $(\cdot)_n:\C\to\C$ \cite[(5.2.4)]{NIST} is defined by
\[
(z)_n:=\prod_{i=1}^n(z+i-1),\nonumber
\]
where $n\in\N_0$. Also, when $z\notin -\N_0$ we have (see \cite[(5.2.5)]{NIST})
\begin{equation}
(z)_n=\frac{\Gamma(z+n)}{\Gamma(z)},
\label{Gammaz+n}
\end{equation}
where $\Gamma:\C\setminus-\N_0\to\C$ is the gamma function (see Chapter 5 in Olver {\it et al.}~(2010)
\cite{NIST}).

\medskip

\noindent Throughout this paper we rely on the following definitions.  
For $a_1,a_2,a_3,\ldots\in\C$, if $i,j\in\Z$ and $j<i$ then
$\sum_{n=i}^{j}a_n=0$ and $\prod_{n=i}^ja_n=1$.
The set of natural numbers is given by $\N:=\{1,2,3,\ldots\}$, the set
$\N_0:=\{0,1,2,\ldots\}=\N\cup\{0\}$, and the set
$\Z:=\{0,\pm 1,\pm 2,\ldots\}.$  The set $\R$ represents the real numbers.

\section{Expansions in Jacobi polynomials}
\label{Jacobifunctiongeneralization}

The Jacobi polynomials $P_n^{(\alpha,\beta)}:\C\to\C$ can be defined 
in terms of a terminating Gauss hypergeometric series as follows
(Olver {\it et al.}~(2010) \cite[(18.5.7)]{NIST})
\[
P_n^{(\alpha,\beta)}(z):=
\frac{(\alpha+1)_n}{n!}\,
{}_2F_1\left(
\begin{array}{c}
-n,n+\alpha+\beta+1\\[0.1cm]
\alpha+1
\end{array};
\frac{1-z}{2}
\right),
\]
for $n\in\N_0$, and 
$\alpha,\beta>-1$ such that if $\alpha,\beta\in(-1,0)$ then $\alpha+\beta+1\neq 0$.
The orthogonality relation for Jacobi polynomials can be found in 
Olver {\it et al.}~(2010) \cite[(18.2.1), (18.2.5), Table 18.3.1]{NIST}
\begin{equation}
\int_{-1}^1
P_m^{(\alpha,\beta)}(x)
P_n^{(\alpha,\beta)}(x)
(1-x)^\alpha
(1+x)^\beta
dx
=
\frac{2^{\alpha+\beta+1}\Gamma(\alpha+n+1)\Gamma(\beta+n+1)}
{(2n+\alpha+\beta+1)\Gamma(\alpha+\beta+n+1)n!}
\delta_{m,n}.
\label{JacobiOrthogonality}
\end{equation}
A connection relation with one free parameter for Jacobi polynomials can 
be found in Olver {\it et al.}~(2010) \cite[(18.18.14)]{NIST}, namely
\begin{eqnarray}
&&\hspace{-1.3cm}P_n^{(\alpha,\beta)}(x)
=\frac{(\beta+1)_n}{(\gamma+\beta+1)(\gamma+\beta+2)_n}\nonumber\\[0.2cm]
&&\hspace{+1.0cm}\times\sum_{k=0}^n \frac{(\gamma+\beta+2k+1)(\gamma+\beta+1)_k\,
(n+\beta+\alpha+1)_k(\alpha-\gamma)_{n-k}}
{(\beta+1)_k\,(n+\gamma+\beta+2)_k(n-k)!}
P_k^{(\gamma,\beta)}(x).
\label{JacConnect}
\end{eqnarray}
In the remainder of the paper, we will use the following
global notation $\RR:=\sqrt{1+\rho^2-2\rho x}$.

\begin{thm}
Let $\alpha\in\C$, $\beta,\gamma>-1$ such that if $\beta,\gamma\in(-1,0)$ then $\beta+\gamma+1\neq 0$,
$\rho\in\{z\in\C:|z|<1\}$, $x\in[-1,1]$. Then
\begin{eqnarray}
&&\hspace{-1.3cm}\frac{2^{\alpha+\beta}}
{\RR\left(1+\RR-\rho\right)^\alpha\left(1+\RR+\rho\right)^\beta}\nonumber\\[0.2cm]
&&\hspace{+0.6cm}=\frac{1}{\gamma+\beta+1}
\sum_{k=0}^\infty
\frac{(2k+\gamma+\beta+1)(\gamma+\beta+1)_k\,
\left(\frac{\alpha+\beta+1}{2}\right)_k\,\left(\frac{\alpha+\beta+2}{2}\right)_k}
{(\alpha+\beta+1)_k\,\left(\frac{\gamma+\beta+2}{2}\right)_k\,\left(\frac{\gamma+\beta+3}{2}\right)_k}\nonumber\\[0.2cm]
&&\hspace{+3.4cm}\times\,{_3}F_2\left(\begin{array}{c}
\beta+k+1,\alpha+\beta+2k+1,\alpha-\gamma\\[0.1cm]
\alpha+\beta+k+1,\gamma+\beta+2k+2\end{array};\rho\right)
\rho^k P_k^{(\gamma,\beta)}(x).
\label{JacGenGen}
\end{eqnarray}
\label{JacGeneralGen}
\end{thm}
\noindent {\bf Proof.} Olver {\it et al.}~(2010) \cite[(18.12.1)]{NIST}
give the generating for Jacobi polynomials, namely
\begin{equation}
\frac{2^{\alpha+\beta}}
{\RR\left(1+\RR-\rho\right)^\alpha\left(1+\RR+\rho\right)^\beta}
=\sum_{n=0}^\infty \rho^n P_n^{(\alpha,\beta)}(x).
\label{JacGenerating}
\end{equation}
This generating function is special in that it is the only 
known algebraic generating function for Jacobi polynomials
(see \cite[p.~90]{Ismail}).
Using the Jacobi connection relation (\ref{JacConnect}) in (\ref{JacGenerating}) produces a double sum.
In order to justify reversing the order of the double summation
expression we show that 
\begin{equation}\label{eq1}
\sum_{n=0}^\infty |c_n|\sum_{k=0}^n |a_{nk}|\left|P_k^{(\alpha,\beta)}(x)\right|< \infty ,
\end{equation}
where $c_n=\rho^n$ and $a_{nk}$ are the connection coefficients satisfying
\[ P_n^{(\alpha,\beta)}(x)=\sum_{k=0}^n a_{nk}P_k^{(\gamma,\beta)}(x) .\] 
We assume that $\alpha,\beta,\gamma>-1$, $x\in[-1,1],$ and $|\rho|<1$.
It follows from \cite[Theorem 7.32.1]{Szego} that 
\begin{equation}\label{eq2}
 \max_{x\in[-1,1]}\left|P_n^{(\alpha,\beta)}(x)\right|\le K_1(1+n)^\sigma ,
\end{equation}
where $K_1$ and $\sigma$ are positive constants.
In order to estimate $a_{nk}$ we use 
\[ a_{nk}=\frac{\int_{-1}^1(1-x)^{\gamma}(1+x)^\beta P_n^{(\alpha,\beta)}(x)P_k^{(\gamma,\beta)}(x)\,dx}{\int_{-1}^1(1-x)^{\gamma}(1+x)^\beta \{P_k^{(\gamma,\beta)}(x)\}^2\,dx} .\]
Using \eqref{eq2} we have
\[ \left|\int_{-1}^1(1-x)^\gamma(1+x)^\beta P_n^{(\alpha,\beta)}(x)P_k^{(\gamma,\beta)}(x)\,dx\right|\le K_2(1+k)^\sigma(1+n)^\sigma .\]
Using (\ref{JacobiOrthogonality}), we get 
\[  \left|\int_{-1}^1(1-x)^{\gamma}(1+x)^\beta \{P_k^{(\gamma,\beta)}(x)\}^2\,dx \right|\ge \frac{K_3}{1+k} \]
where $K_3>0$.
Therefore, 
\[ |a_{nk}|\le K_4 (1+k)^{\sigma+1}(1+n)^\sigma .\]
Finally, we show \eqref{eq1}:
\begin{eqnarray*}
 \sum_{n=0}^\infty |c_n|\sum_{k=0}^n |a_{nk}|
 \left|P_k^{(\alpha,\beta)}(x)\right|&\le &K_5 \sum_{n=0}^\infty |\rho|^n 
 \sum_{k=0}^n (1+k)^{2\sigma+1}(1+n)^\sigma \\
&\le & K_5 \sum_{n=0}^\infty |\rho|^n (1+n)^{3\sigma+2}<\infty,
\end{eqnarray*}
because $|\rho|<1$.
Reversing the order of the summation and shifting the $n$-index by $k$
with simplification, and by analytic continuation in $\alpha$, 
we produce the generalization (\ref{JacGenGen}).
$\hfill\blacksquare$

\begin{thm}
Let $\alpha\in\C$, $\beta,\gamma>-1$ such that if $\beta,\gamma\in(-1,0)$ then $\beta+\gamma+1\neq 0$,
$\rho\in\{z\in\C:|z|<1\}$, $x\in(-1,1)$. Then
\begin{eqnarray}
&&\hspace{-0.6cm}\left(\frac{2}{(1-x)\rho}\right)^{\alpha/2}\left(\frac{2}{(1+x)\rho}\right)^{\beta/2}
J_\alpha\left(\sqrt{2(1-x)\rho}\right)I_\beta\left(\sqrt{2(1+x)\rho}\right)\nonumber\\[0.2cm]
&&\hspace{-0.0cm}=\frac{1}{(\gamma+\beta+1)\Gamma(\alpha+1)\Gamma(\beta+1)}
\sum_{k=0}^\infty\frac{(2k+\gamma+\beta+1)(\gamma+\beta+1)_k\,
\left(\frac{\alpha+\beta+1}{2}\right)_k\,\left(\frac{\alpha+\beta+2}{2}\right)_k}
{(\alpha+1)_k\,(\beta+1)_k\,(\alpha+\beta+1)_k\,
\left(\frac{\gamma+\beta+2}{2}\right)_k\,\left(\frac{\gamma+\beta+3}{2}\right)_k}\nonumber\\[0.2cm]
&&\hspace{+3.0cm}\times\,{_2}F_3\left(\begin{array}{c}
2k+\alpha+\beta+1,\alpha-\gamma\\[0.1cm]
\alpha+\beta+k+1,\gamma+\beta+2k+2,\alpha+k+1\end{array};\rho\right)\rho^kP_k^{(\gamma,\beta)}(x).
\label{Jac2F3}
\end{eqnarray}
\label{ThmJac2F3}
\end{thm}
\noindent {\bf Proof.}
Olver {\it et al.} \cite[(18.12.2)]{NIST} give a generating function for Jacobi polynomials, namely
\begin{eqnarray}
&&\hspace{-1.1cm}\left(\frac{2}{(1-x)\rho}\right)^{\alpha/2}\left(\frac{2}{(1+x)\rho}\right)^{\beta/2}
J_\alpha\left(\sqrt{2(1-x)\rho}\right)I_\beta\left(\sqrt{2(1+x)\rho}\right)\nonumber\\[0.2cm]
&&\hspace{+6.2cm}=\sum_{n=0}^\infty\frac{1}{\Gamma(\alpha+1+n)\Gamma(\beta+1+n)}
\rho^nP_n^{(\alpha,\beta)}(x).
\label{jacbessel}
\end{eqnarray}
Using the connection relation for Jacobi polynomials (\ref{JacConnect}) in (\ref{jacbessel})
produces a double sum. Reversing the order of the summation and shifting the $n$-index by $k$
with simplification produces this generalization for a generating function of
Jacobi polynomials.
$\hfill\blacksquare$

\begin{defn}
A {\bf companion identity} is one which is produced by applying the
map $x\mapsto-x$ to an expansion over
Jacobi polynomials or in terms of those
orthogonal polynomials which can be obtained as
limiting cases of Jacobi
polynomials (i.e., Gegenbauer, Chebyshev, and Legendre polynomials)
with argument $x$
in conjunction with the parity relations for those orthogonal polynomials.
\end{defn}

By starting with (\ref{JacGenGen}) and 
(\ref{Jac2F3}), applying the parity relation
for Jacobi polynomials 
(see for instance Olver {\it et al.}~(2010) \cite[Table 18.6.1]{NIST}) 
\[
P_n^{(\alpha,\beta)}(-x)=(-1)^n P_n^{(\beta,\alpha)}(x),
\]
and mapping $\rho\mapsto-\rho$,
one obtains the corresponding companion identities.
Although for (\ref{Jac2F3}), one must substitute $-1=e^{\pm i\pi}$, 
and use Olver {\it et al.}~(2010) \cite[(10.27.6)]{NIST}.
Therefore Theorems \ref{JacGeneralGen} and \ref{ThmJac2F3}
are valid when the left-hand sides remain
the same, and on the right-hand sides $\alpha,\beta\mapsto\beta,\alpha$,
the arguments of the ${}_3F_2$ and ${}_2F_3$ are replaced by $-\rho$, and 
the order of the Jacobi polynomials become $(\alpha,\gamma)$.  

\begin{thm}
Let $\alpha\in\C$, $\beta,\gamma>-1$ such that if $\beta,\gamma\in(-1,0)$ then $\beta+\gamma+1\neq 0$,
$\rho\in\{z\in\C:|z|<1\}\setminus(-1,0]$, $x\in[-1,1]$. Then
\begin{eqnarray}
&&\hspace{-0.9cm}\frac{(1+x)^{-\beta/2}}
{\RR^{\alpha+1}}
P_\alpha^{-\beta}\left(\frac{1+\rho}{\RR}\right)
=\frac{\Gamma(\gamma+\beta+1)}
{2^{\beta/2}\Gamma(\beta+1)(1-\rho)^{\alpha-\gamma}\rho^{(\gamma+1)/2}}\nonumber\\[0.2cm]
&&\hspace{+0.5cm}\times\sum_{k=0}^\infty\frac{(2k+\gamma+\beta+1)(\gamma+\beta+1)_k\,(\alpha+\beta+1)_{2k}}{(\beta+1)_k}
P_{\gamma-\alpha}^{-\gamma-\beta-2k-1}\left(\frac{1+\rho}{1-\rho}\right)
P_k^{(\gamma,\beta)}(x).
\label{Jacwithalpha}
\end{eqnarray}
\label{ultrajac}
\end{thm}
\noindent {\bf Proof.} Olver {\it et al.}~(2010) \cite[(18.12.3)]{NIST} give a 
generating function for Jacobi polynomials, namely
\begin{equation}
(1+\rho)^{-\alpha-\beta-1}\,{_2}F_1\left(\begin{array}{c}
\frac12(\alpha+\beta+1),\frac12(\alpha+\beta+2)\\[0.2cm]
\beta+1\end{array};\frac{2(1+x)\rho}{(1+\rho)^2}\right)
=\sum_{n=0}^\infty\frac{(\alpha+\beta+1)_n}{(\beta+1)_n}\rho^nP_n^{(\alpha,\beta)}(x).
\label{DLMF3}
\end{equation}
Using the connection relation for Jacobi polynomials (\ref{JacConnect}) in (\ref{DLMF3})
produces a double sum on the right-hand side of the equation. Reversing the order of the summation
and shifting the $n$-index by $k$ with simplification gives a Gauss hypergeometric function as the coefficient
of the expansion. The resulting expansion formula is 
\begin{eqnarray}
&&\hspace{-1.1cm}(1+\rho)^{-\alpha-\beta-1}\,{_2}F_1\left(\begin{array}{c}
\frac12(\alpha+\beta+1),\frac12(\alpha+\beta+2)\\[0.2cm]
\beta+1\end{array};\frac{2(1+x)\rho}{(1+\rho)^2}\right)\nonumber\\[0.2cm]
&&\hspace{+0.5cm}=\frac{1}{\gamma+\beta+1}\sum_{k=0}^\infty\frac{(2k+\gamma+\beta+1)(\gamma+\beta+1)_k\,
\left(\frac{\alpha+\beta+1}{2}\right)_k\,\left(\frac{\alpha+\beta+2}{2}\right)_k}
{(\beta+1)_k\,\left(\frac{\gamma+\beta+2}{2}\right)_k\,\left(\frac{\gamma+\beta+3}{2}\right)_k}\nonumber\\[0.2cm]
&&\hspace{+5.0cm}\times\,{_2}F_1\left(\begin{array}{c}
\alpha+\beta+1+2k,\alpha-\gamma\\[0.2cm]
\gamma+\beta+2+2k\end{array};\rho\right)\rho^kP_k^{(\gamma,\beta)}(x).
\label{DLMF3afterconnection}
\end{eqnarray}
The Gauss hypergeometric function coefficient is realized to be an associated Legendre function of the first kind.
The associated Legendre function of the first kind $P_\nu^\mu :\C\setminus(-\infty,1]\to\C$
(see Chapter 14 in Olver {\it et al.}~(2010) \cite{NIST}) can be defined in terms of the
Gauss hypergeometric function as follows (Olver {\it et al.}~(2010) \cite[(14.3.6),(15.2.2), \S 14.21(i)]{NIST})
\begin{equation}
P_\nu^\mu(z):=\frac{1}{\Gamma(1-\mu)}\left(\frac{z+1}{z-1}\right)^{\mu/2}\,{_2}F_1\left(\begin{array}{c}
-\nu,\nu+1\\[0.1cm]1-\mu\end{array};\frac{1-z}{2}\right),
\label{LegP}
\end{equation}
for $z\in\C\setminus(-\infty,1]$.
Using a relation for the Gauss hypergeometric function from Olver {\it et al.}~(2010) \cite[(15.9.19)]{NIST}, namely
\begin{equation}
{_2}F_1\left(\begin{array}{c}a,b\\[0.2cm]a-b+1\end{array};z\right)
=\frac{z^{(b-a)/2}\,\Gamma(a-b+1)}{(1-z)^b}P_{-b}^{b-a}\left(\frac{1+z}{1-z}\right),
\label{simple2F1relation}
\end{equation}
for $z\in\C\setminus\left\{(-\infty,0]\cup(1,\infty)\right\}$,
the Gauss hypergeometric function coefficient of the expansion can be 
expressed as an
associated Legendre function of the first kind.
The Gauss hypergeometric function on the left-hand side of (\ref{DLMF3afterconnection}) can also be 
expressed in terms of the associated Legendre function of the first kind using
Magnus {\it et al.}~(1966) \cite[p.~157, entry 11]{MOS}, namely
\[
P_\nu^\mu(z)=\frac{2^\mu z^{\nu+\mu}}{\Gamma(1-\mu)(z^2-1)^{\mu/2}}\,{_2}F_1\left(\begin{array}{c}
\frac{-\nu-\mu}{2},\frac{-\nu-\mu+1}{2}\\[0.2cm]1-\mu\end{array};1-\frac{1}{z^2}\right),
\]
where $\Re z>0$. This completes the proof.
$\hfill\blacksquare$

\begin{cor}
Let $\beta\in\C$, $\alpha,\gamma>-1$ such that if $\alpha,\gamma\in(-1,0)$ then $\alpha+\gamma+1\neq 0$,
$\rho\in(0,1)$, $x\in[-1,1]$. Then
\begin{eqnarray}
&&\hspace{-0.9cm}\frac{(1-x)^{-\alpha/2}}
{\RR^{\beta+1}}
{\mathrm P}_\beta^{-\alpha}\left(\frac{1-\rho}{\RR}\right)
=\frac{\Gamma(\gamma+\alpha+1)}
{2^{\alpha/2}\Gamma(\alpha+1)(1+\rho)^{\beta-\gamma}\rho^{(\gamma+1)/2}}\nonumber\\[0.2cm]
&&\hspace{+0.5cm}\times\sum_{k=0}^\infty
\frac{(2k+\gamma+\alpha+1)(\gamma+\alpha+1)_k\,(\alpha+\beta+1)_{2k}}{(\alpha+1)_k}
{\mathrm P}_{\gamma-\beta}^{-\gamma-\alpha-2k-1}\left(\frac{1-\rho}{1+\rho}\right)
P_k^{(\alpha,\gamma)}(x).
\label{Jacwithalphacom}
\end{eqnarray}
\label{othercooljac}
\end{cor}
\noindent {\bf Proof.}  
Applying the parity relation for Jacobi polynomials to (\ref{DLMF3afterconnection}) and mapping $\rho\mapsto-\rho$
produces its companion identity.  The Gauss hypergeometric functions
appearing in this expression are Ferrers functions of the first kind 
(often referred to as the associated Legendre function of the first kind on the cut).
The Ferrers function of the first kind
${\mathrm P}_\nu^\mu:(-1,1)\to\C$. 
can be defined in terms of the 
Gauss hypergeometric function as 
follows (Olver {\it et al.}~(2010) \cite[(14.3.1)]{NIST})
\begin{equation}
{\mathrm P}_\nu^\mu(x):=\frac{1}{\Gamma(1-\mu)}\left(\frac{1+x}{1-x}\right)^{\mu/2}\,{_2}F_1\left(\begin{array}{c}
-\nu,\nu+1\\[0.1cm]1-\mu\end{array};\frac{1-x}{2}\right),
\label{ferrers}
\end{equation}
for $x\in(-1,1)$.
The Gauss hypergeometric function coefficient of the expansion is
seen to be a Ferrers function of the first kind by starting 
with (\ref{ferrers}) and using the linear transformation for the
Gauss hypergeometric function
\cite[(15.8.1)]{NIST}. This derives a Gauss hypergeometric
function representation of the Ferrers function of the first kind, namely
\[
{_2}F_1\left(\begin{array}{c}a,b\\[0.2cm]a-b+1\end{array};-x\right)
=\frac{x^{(b-a)/2}\,\Gamma(a-b+1)}{(1+x)^b}{\mathrm P}_{-b}^{b-a}\left(\frac{1-x}{1+x}\right),
\]
for $x\in(0,1)$.  
The Gauss hypergeometric function on the
left-hand side of the companion identity for (\ref{DLMF3afterconnection})
is shown to be a Ferrers function of the first kind through 
Magnus {\it et al.}~(1996) \cite[p.~167]{MOS}, namely
\[
{\mathrm P}_\nu^\mu(x)=\frac{2^\mu x^{\nu+\mu}}{\Gamma(1-\mu)(1-x^2)^{\mu/2}}\,{_2}F_1\left(\begin{array}{c}
\frac{-\nu-\mu}{2},\frac{-\nu-\mu+1}{2}\\[0.2cm]1-\mu\end{array};1-\frac{1}{x^2}\right),
\]
for $x\in(0,1)$. This completes the proof.
$\hfill\blacksquare$

\medskip

The above two theorems are interesting results, as they are general cases for other 
generalizations found from related generating functions. For instance, by applying the connection 
relation for Jacobi polynomials (\ref{JacConnect}) to an important extension of the generating 
function (\ref{DLMF3}) (with the parity relation for Jacobi polynomials applied) given by 
Ismail (2005) \cite[(4.3.2)]{Ismail}, namely 
\begin{eqnarray}
&&\hspace{-1.0cm}\frac{1+\rho}{(1-\rho)^{\alpha+\beta+2}}\,{_2}F_1\left(\begin{array}{c}
\frac{\alpha+\beta+2}{2},\frac{\alpha+\beta+3}{2}\\[0.2cm]
\alpha+1\end{array};\frac{2\rho(x-1)}{(1-\rho)^2}\right)\nonumber\\[0.2cm]
&&\hspace{+5.0cm}=\sum_{n=0}^\infty\frac{(\alpha+\beta+1+2n)(\alpha+\beta+1)_n}{(\alpha+\beta+1)(\alpha+1)_n}
\rho^nP_n^{(\alpha,\beta)}(x),
\label{ismails}
\end{eqnarray}
produces a generalization that is equivalent to mapping $\alpha\mapsto\alpha+1$ in Theorem \ref{ultrajac}.

It is interesting that trying to generalize 
(\ref{ismails}) using the connection relation (\ref{JacConnect}), does not produce a new generalized
formula, since (\ref{Jacwithalpha}) is its generalization --
Ismail proves (\ref{ismails}) by multiplying (\ref{DLMF3})
by $\rho^{(\alpha+\beta+1)/2}$, after the companion identity is applied, 
and then differentiating by $\rho$.
Ismail also mentions that (\ref{ismails}) (and therefore
Theorem \ref{ultrajac} and 
Corollary \ref{othercooljac})
are closely connected to the Poisson 
kernel of $\bigl\{P_n^{(\alpha,\beta)}(x)\bigr\}$.
One can also see that these expansions are related to 
the translation operator associated with Jacobi polynomials
by mapping $\alpha\mapsto\alpha+1$ in Theorem \ref{ultrajac}.

Theorem \ref{ultrajac} and 
Corollary \ref{othercooljac}
are also generalizations of the expansion
(see Cohl \& MacKenzie (2013) \cite{CohlMacK12})
\begin{eqnarray}
&&\hspace{-1.0cm}\frac{(1+x)^{-\beta/2}}
{\RR^{\alpha+m+1}}
P_{\alpha+m}^{-\beta}\left(\frac{1+\rho}{\RR}\right)\nonumber\\[0.2cm]
&&\hspace{+1.0cm}=\frac{\rho^{-(\alpha+1)/2}}{2^{\beta/2}(1-\rho)^m}\sum_{n=0}^\infty
\frac{(2n+\alpha+\beta+1)\Gamma(\alpha+\beta+n+1)(\alpha+\beta+m+1)_{2n}}{\Gamma(\beta+n+1)}\nonumber\\[0.2cm]
&&\hspace{+8.0cm}\times P_{-m}^{-\alpha-\beta-2n-1}\left(\frac{1+\rho}{1-\rho}\right)P_n^{(\alpha,\beta)}(x),
\label{withm}
\end{eqnarray}
found by mapping $\alpha,\gamma\mapsto\alpha+m,\alpha$ for $m\in\N_0$ in Theorem \ref{ultrajac};
and its companion identity (see Cohl \& MacKenzie (2013) \cite{CohlMacK12}),
\begin{eqnarray}
&&\hspace{-1.0cm}\frac{(1-x)^{-\alpha/2}}
{\RR^{\beta+m+1}}
{\mathrm P}_{\beta+m}^{-\alpha}\left(\frac{1-\rho}{\RR}\right)\nonumber\\[0.2cm]
&&\hspace{+1.0cm}=\frac{\rho^{-(\beta+1)/2}}{2^{\alpha/2}(1+\rho)^m}\sum_{n=0}^\infty
\frac{(2n+\alpha+\beta+1)\Gamma(\alpha+\beta+n+1)(\alpha+\beta+m+1)_{2n}}{\Gamma(\alpha+n+1)}\nonumber\\[0.2cm]
&&\hspace{+8.0cm}\times {\mathrm P}_{-m}^{-\alpha-\beta-2n-1}\left(\frac{1-\rho}{1+\rho}\right)
P_n^{(\alpha,\beta)}(x),
\label{withmcom}
\end{eqnarray}
found by mapping $\beta,\gamma\mapsto\beta+m,\beta$ for $m\in\N_0$ in Theorem \ref{othercooljac}.
The expansions (\ref{withm}), (\ref{withmcom}) are produced 
using the definition of 
the Gauss hypergeometric function on the left hand side of (\ref{ismails}) and the expansion of
$(1-x)^n$ in terms of Jacobi polynomials (see Cohl \& MacKenzie (2013) \cite[(7), (13)]{CohlMacK12}). 
Interestingly, the expansions (\ref{withm}) and (\ref{withmcom}) are also related to the
generalized translation operator, but with a more general translation that can be seen with the
$\alpha\mapsto\alpha+m$ or $\beta\mapsto\beta+m$.

\section{Expansions in Gegenbauer polynomials}
\label{Gegenbauerpolynomials}

The Gegenbauer polynomials $C_n^{\mu}:\C\to\C$ can be defined 
in terms of the terminating Gauss hypergeometric series as follows
(Olver {\it et al.}~(2010) \cite[(18.5.9)]{NIST})
\[
C_n^{\mu}(z):=\frac{(2\mu)_n}{n!}\,{}_2F_1\left(
\begin{array}{c}
-n,n+2\mu\\[0.1cm]
\mu+\frac12
\end{array};
\frac{1-z}{2}
\right),
\]
for $n\in\N_0$ and 
$\mu\in(-1/2,\infty)\setminus\{0\}$.
The orthogonality relation for Gegenbauer polynomials can be found in 
Olver {\it et al.}~(2010) \cite[(18.2.1), (18.2.5), Table 18.3.1]{NIST}
for $m,n\in\N_0$,
namely
\begin{equation}
\int_{-1}^1
C_m^{\mu}(x)
C_n^{\mu}(x)
(1-x^2)^{\mu-1/2}
dx
=
\frac{\pi 2^{1-2\mu}\Gamma(2\mu+n)}{(n+\mu)\Gamma^2(\mu)n!}
\delta_{m,n}.
\label{OrthoGegen}
\end{equation}

\begin{thm} Let $\lambda,\mu\in\C$, $\nu\in(-1/2,\infty)\setminus\{0\},$
$\rho\in\left\{z\in\C:|z|<1\right\},$ $x\in[-1,1]$. Then
\begin{eqnarray}
&&\hspace{-0.5cm}(1-x^2)^{1/4-\mu/2}
P_{\lambda+\mu-1/2}^{1/2-\mu}\left(\RR+\rho\right)\,
{\mathrm P}_{\lambda+\mu-1/2}^{1/2-\mu}\left(\RR-\rho\right)\nonumber\\[0.2cm]
&&\hspace{+0.5cm}=\frac{(\rho/2)^{\mu-1/2}}{\nu\,\Gamma^2(\mu+1/2)}
\sum_{n=0}^\infty \frac{(\nu+n)(-\lambda)_n\,(2\mu+\lambda)_n\,(\mu)_n}
{(2\mu)_n\,(\mu+1/2)_n\,(\nu+1)_n}\nonumber\\[0.2cm]
&&\hspace{+1.5cm}\times\,{_6}F_5\left(
\begin{array}{c}
\frac{-\lambda+n}{2},\frac{-\lambda+n+1}{2},\frac{2\mu+\lambda+n}{2},
\frac{2\mu+\lambda+n+1}{2}
,\mu+n
,\mu-\nu
\\[0.2cm]
\frac{2\mu+n}{2},\frac{2\mu+n+1}{2},\frac{\mu+n+\frac12}{2},\frac{\mu+n+\frac32}{2},\nu+1+n
\end{array};\rho^2\right)\rho^{n}C_n^\nu(x).
\label{Gegenwith6F5}
\end{eqnarray}
\end{thm}
\noindent {\bf Proof.}
In Koekoek, Lesky \& Swarttouw (2010) 
\cite[(9.8.32)]{Koekoeketal}, there is the following 
generating function for Gegenbauer polynomials
\begin{eqnarray}
{_2}F_1\left(\begin{array}{c}
\lambda,2\mu-\lambda\\[0.2cm]
\mu+\frac12\end{array};\frac{1-\rho-\RR}{2}\right)
\,{_2}F_1\left(\begin{array}{c}
\lambda,2\mu-\lambda\\[0.2cm]
\mu+\frac12\end{array};\frac{1+\rho-\RR}{2}\right)
=\sum_{n=0}^\infty \frac{(\lambda)_n\,(2\mu-\lambda)_n}
{(2\mu)_n\,(\mu+\frac12)_n} \rho^n C_n^\mu(x).\nonumber
\end{eqnarray}
These Gauss hypergeometric functions can be re-written in terms of 
associated Legendre and Ferrers functions of the first kind.
The first Gauss hypergeometric function can be written in terms
the Ferrers function of the first kind using Abramowitz \& 
Stegun (1972) \cite[(15.4.19)]{Abra}, namely
\[
{}_2F_1\left(
\begin{array}{c}
a,b\\[0.1cm]\frac{a+b+1}{2}
\end{array};x\right)=\Gamma
\left(\frac{a+b+1}{2}\right)\left(x(1-x)\right)^{(1-a-b)/4}
{\rm P}_{(a-b-1)/2}^{(1-a-b)/2}(1-2x),
\]
for $x\in(0,1),$ with $a=\lambda$, $b=2\mu-\lambda$.
The second Gauss hypergeometric function can
be written in terms of the associated Legendre function
of the first kind using \cite[(14.3.6)]{NIST} and Euler's 
transformation \cite[(15.8.1)]{NIST}.  The substitutions yield
\begin{eqnarray}
&&\hspace{-1cm}(1-x^2)^{1/4-\mu/2}
P_{\lambda+\mu-1/2}^{1/2-\mu}(\RR+\rho)
{\rm P}_{\lambda+\mu-1/2}^{1/2-\mu}(\RR-\rho)\nonumber\\[0.2cm]
&&\hspace{5cm}=\frac{(\rho/2)^{\mu-1/2}}{\Gamma^2(\mu+1/2)}
\sum_{n=0}^\infty 
\frac{(-\lambda)_n(2\mu+\lambda)_n}{(2\mu)_n(\mu+1/2)_n}
\rho^nC_n^\mu(x).
\label{genfngegafterPPtrans}
\end{eqnarray}

Using the connection relation for Gegenbauer 
polynomials (\ref{GegenConnect})
on the generating function (\ref{genfngegafterPPtrans})
produces a double sum. 
Reversing the order of the summation and shifting the $n$-index by $2k$
with simplification completes the proof.
$\hfill\blacksquare$

\medskip

Associated Legendre and Ferrers functions of the first kind
with special values of the degree and order reduce to 
Gegenbauer polynomials.  For instance, if 
$n\in\N_0$, then through
\cite[(14.3.22)]{NIST}
\[
P_{n+\mu-1/2}^{1/2-\mu}(z)=
\frac{2^{\mu-1/2}\Gamma(\mu)n!}
{\sqrt{\pi}\,\Gamma(2\mu+n)}
(z^2-1)^{\mu/2-1/4}\,C_n^\mu(z),
\]
and from \cite[(14.3.21)]{NIST}, one has
\[
{\mathrm P}_{n+\mu-1/2}^{1/2-\mu}(x)=
\frac{2^{\mu-1/2}\Gamma(\mu)n!}
{\sqrt{\pi}\,\Gamma(2\mu+n)}
(1-x^2)^{\mu/2-1/4}\,C_n^\mu(x).
\]
From (\ref{genfngegafterPPtrans}) using the above expressions,
we have the following finite-summation generating function
expression with $m\in\N_0$,
\begin{equation}
C_m^\mu(\RR+\rho)C_m^\mu(\RR-\rho)=\frac{(2\mu)_m^2}
{(m!)^2}\sum_{n=0}^m
\frac{(-m)_n(2\mu+m)_n}{(2\mu)_n(\mu+\frac12)_n}\rho^nC_n^\mu(x),
\label{prodgeggen}
\end{equation}
and from the generalized result
(\ref{Gegenwith6F5}) we have
\begin{eqnarray*}
&&\hspace{-0.5cm}
C_m^\mu(\RR+\rho)C_m^\mu(\RR-\rho)=\frac{(2\mu)_m^2}{\nu(m!)^2}
\sum_{n=0}^m \frac{(\nu+n)\,(-m)_n\,(2\mu+m)_n\,(\mu)_n}
{(2\mu)_n\,(\mu+1/2)_n\,(\nu+1)_n}\nonumber\\[0.2cm]
&&\hspace{+1.5cm}\times\,{_6}F_5\left(
\begin{array}{c}
\frac{-m+n}{2},\frac{-m+n+1}{2},\frac{2\mu+m+n}{2},\frac{2\mu+m+n+1}{2},\mu-\nu,\mu+n\\[0.2cm]
\frac{2\mu+n}{2},\frac{2\mu+n+1}{2},\frac{\mu+n+\frac12}{2},\frac{\mu+n+\frac32}{2},\nu+1+n
\end{array};\rho^2\right)\rho^{n}C_n^\nu(x),
\end{eqnarray*}
which reduces to 
(\ref{prodgeggen}) when $\nu=\mu$.

\medskip

Consider the generating function for Gegenbauer polynomials,
Olver {\it et al.}~(2010) \cite[(18.12.5)]{NIST} 
\begin{equation}
\frac{1-\rho x}{(1+\rho^2-2\rho x)^{\nu+1}}
=\frac{1}{2\nu}\sum_{n=0}^\infty (n+2\nu)\rho^n C_n^\nu(x),
\label{GegenGen5}
\end{equation}
and the generating function
\begin{equation}
\frac{x-\rho}{(1+\rho^2-2\rho x)^{\nu+1}}
=\frac{1}{2\nu\rho}
\sum_{n=0}^\infty n\rho^nC_n^\nu(x),
\label{changedGegenGen}
\end{equation}
which follows from (\ref{GegenGen5}) using (\ref{generatingfnforgeg}).
The technique of this paper can also be applied to generalize 
(\ref{GegenGen5}) and
(\ref{changedGegenGen}).
However, note that
\[
\frac{1-\rho x}{(1+\rho^2-2\rho x)^{\nu+1}}=\frac{1-\rho^2}{2} \frac{1}{(1+\rho^2-2\rho x)^{\nu+1}}+\frac12\frac{1}{(1+\rho^2-2\rho x)^\nu},
\]
\[
 \frac{x-\rho}{(1+\rho^2-2\rho x)^{\nu+1}}=\frac{1-\rho^2}{2\rho} \frac{1}{(1+\rho^2-2\rho x)^{\nu+1}}-\frac1{2\rho}\frac{1}{(1+\rho^2-2\rho x)^\nu},
\]
so it is easier to use 
(\ref{generatingfnforgeggen})
on the right-hand sides.  

\medskip

The Gegenbauer polynomials can be defined as a specific case of the Jacobi polynomials, namely
\begin{equation}
C_n^\nu(x)=\frac{(2\nu)_n}{\left(\nu+\frac12\right)_n}P_n^{(\nu-1/2,\nu-1/2)}(x).
\label{JactoGeg}
\end{equation}
Therefore the expansions given in the section on Jacobi polynomials can 
also be written as expansions in Gegenbauer polynomials by using symmetric 
parameters.
Furthermore, these expansions can also be written as expansions over 
Chebyshev polynomials of the second kind and Legendre polynomials using
$U_n(z)=C_n^1(z),$
$P_n(z)=C_n^{1/2}(z),$
for $n\in\N_0$.
One may also take the limit of an expansion in Gegenbauer 
polynomials as $\mu\to0.$ This limit may be 
well defined with the interpretation of obtaining 
Chebyshev polynomials of the first 
kind through
Andrews {\it et al.}~(1999) \cite[(6.4.13)]{AAR}, namely
\begin{equation}
T_n(z)=\frac{1}{\epsilon_n}\lim_{\mu\to 0}\frac{n+\mu}{\mu}C_n^\mu(z),
\label{defnChebyshev1intermsofGegenbauer}
\end{equation}
where the Neumann factor $\epsilon_n\in\{1,2\},$ defined by 
$\epsilon_n:=2-\delta_{n,0},$ commonly seen in Fourier 
cosine series.
We can, for example, derive the following corollaries.

\begin{cor}
Let 
$\alpha\in\C$, $\gamma\in(-1/2,\infty)\setminus\{0\}$,
$\rho\in\{z\in\C:|z|<1\}$, $x\in[-1,1]$. Then
\begin{eqnarray}
&&\hspace{-1.0cm}\frac{2^{\alpha+\gamma-1}}
{\RR\left(1+\RR-\rho\right)^{\alpha-1/2}\left(1+\RR+\rho\right)^{\gamma-1/2}}\nonumber\\[0.2cm]
&&\hspace{-0.4cm}=\frac{1}{\gamma}\sum_{k=0}^\infty
\frac{(k+\gamma)\left(\frac{\alpha+\gamma}{2}\right)_k\,\left(\frac{\alpha+\gamma+1}{2}\right)_k}
{(\alpha+\gamma)_k\,(\gamma+1)_k}\,
{_3}F_2\left(\begin{array}{c}
\gamma+k+\frac12,\alpha+\gamma+2k,\alpha-\gamma\\[0.1cm]
\alpha+\gamma+k,2\gamma+2k+1\end{array};\rho\right)\rho^k C_k^\gamma(x).
\label{GegenFromJacGen}
\end{eqnarray}
\end{cor}
\noindent {\bf Proof.} Using (\ref{JacGenGen}), mapping $\alpha\mapsto\alpha-1/2$ and
$\beta$, $\gamma\mapsto\gamma-1/2$, and using (\ref{JactoGeg}) completes the proof.
$\hfill\blacksquare$

\begin{cor}
Let 
$\alpha\in\C$,
$\rho\in\{z\in\C:|z|<1\}$, $x\in[-1,1]$. Then
\begin{equation}
\frac{(1+\RR+\rho)^{1/2}}{\RR(1+\RR-\rho)^{\alpha-1/2}}
=2^{1-\alpha}\sum_{k=0}^\infty\epsilon_k\frac{\left(\frac\alpha2\right)_k\,\left(\frac{\alpha+1}{2}\right)_k}
{(\alpha)_k\,k!}
\,{_3}F_2\left(\begin{array}{c}
k+\frac12,\alpha+2k,\alpha\\[0.1cm]
2k+1,\alpha+k\end{array};\rho\right)\rho^k\,T_k(x).
\label{ChebyT3F2}
\end{equation}
\end{cor}
\noindent {\bf Proof.}
Taking the limit as $\gamma\to0$ of (\ref{GegenFromJacGen}) and using
(\ref{defnChebyshev1intermsofGegenbauer}) completes the proof.
$\hfill\blacksquare$

\section{Expansions in Laguerre polynomials}
\label{Laguerrepolynomials}

\noindent The Laguerre polynomials $L_n^\alpha : \C \to \C$ can be defined in terms of Kummer's
confluent hypergeometric function of the first kind as follows (Olver {\it et al.}~(2010) \cite[(18.5.12)]{NIST})
\[
L_n^\alpha(z):=\frac{(\alpha+1)_n}{n!} 
M(-n,\alpha+1,z),\nonumber
\]
for $n\in\N_0$, and $\alpha > -1$.
The Laguerre function $L_\nu^\alpha:\C\to\C,$ which generalizes the Laguerre polynomials
is defined as follows 
(Erd\'{e}lyi {\it et al.}~(1981) \cite[(6.9.2.37), this equation is stated incorrectly therein]{Erdelyi})
for $\nu,\alpha\in\C$,
\begin{equation}
L_\nu^\alpha(z):=\frac{\Gamma(1+\nu+\alpha)}{\Gamma(\nu+1)\Gamma(\alpha+1)}M(-\nu,\alpha+1,z).
\label{Laguerredefn}
\end{equation}
The orthogonality relation for Laguerre polynomials can be found in Olver {\it et al.}~(2010) \cite[(18.2.1), (18.2.5), Table 18.3.1]{NIST}
\begin{equation}
\int_0^\infty x^\alpha e^{-x} L_n^\alpha(x) L_m^\alpha(x)dx=\frac{\Gamma(n+\alpha+1)}{n!} \delta_{n,m}.
\label{OrthoLag}
\end{equation}
The connection relation for Laguerre polynomials, given by Olver {\it et al.}~(2010) \cite[(18.18.18)]{NIST} (see also Ruiz \& Dehesa (2001) \cite{SanchezDehesa}),
is
\begin{equation}
L_n^\alpha(x)=\sum_{k=0}^n \frac{(\alpha-\beta)_{n-k}}{(n-k)!} L_k^\beta(x).
\label{LagConnect}
\end{equation}

\begin{thm}
Let $\alpha,\beta\in\R,$ $x> 0$, $\rho\in\C.$ Then
\begin{equation}
x^{-\alpha/2}  J_\alpha\left(2\sqrt{x\rho}\right)=
\rho^{\alpha/2} e^{-\rho} \sum_{k=0}^\infty\frac{\Gamma(\beta-\alpha+1)}{\Gamma(\beta+1+k)}
L_{\beta-\alpha}^{\alpha+k}(\rho)\rho^k L_k^\beta(x).
\label{LagforInt1}
\end{equation}
\end{thm}
\noindent {\bf Proof.}
Olver {\it et al.}~(2010) \cite[(18.12.14)]{NIST} give a generating function for Laguerre polynomials, namely
\begin{equation*}
x^{-\alpha/2} J_\alpha(2\sqrt{x\rho})=
\rho^{\alpha/2}e^{-\rho}\sum_{n=0}^\infty \frac{\rho^n}{\Gamma(\alpha+1+n)}L_n^\alpha(x),
\end{equation*}
where $J_\alpha$ is the Bessel function of the first kind (\ref{BesselJ}).
Using the Laguerre connection relation (\ref{LagConnect}) 
to replace the Laguerre polynomial in the generating 
function produces a double sum. 
In order to justify reversing the resulting order 
of summation, we demonstrate that
\begin{equation}
\sum_{n=0}^\infty |c_n|\sum_{k=0}^n |a_{nk}| \left|L_k^\beta(x)\right|<\infty, 
\label{doublesumLaguerre}
\end{equation}
where
\[ c_n=\frac{\rho^n}{\Gamma(\alpha+1+n)}\]
and
\begin{equation}
a_{nk}=\frac{(\alpha-\beta)_{n-k}}{(n-k)!}.
\label{ank}
\end{equation}
We assume that $\alpha,\beta\in\R$, $\rho\in\C$ and $x>0$.
It is known \cite[Theorem 8.22.1]{Szego} that
\begin{equation}
\bigl|L_n^\alpha(x)\bigr|
\le K_1(1+n)^{\sigma_1},
\label{Laguerrebound}
\end{equation}
where $K_1,\sigma_1=\frac{\alpha}{2}-\frac14$ are constants independent of $n$ (but depend on $x$ and $\alpha$).
We also have
\begin{equation}
|a_{nk}|\le (1+n-k)^{\sigma_2} \le (1+n)^{\sigma_2},
\label{ankbound}
\end{equation}
where $\sigma_2=|\alpha-\beta|$.
Therefore,
\[
\sum_{n=0}^\infty |c_n|\sum_{k=0}^n |a_{nk}| \left|L_k^\beta(x)\right|
\le K_1\sum_{n=0}^\infty \frac{|\rho|^n}{\Gamma(\alpha+1+n)} 
(1+n)^{\sigma_1+\sigma_2+1} <\infty .
\]
Reversing the order of summation and shifting 
the $n$-index by $k$ yields
\[
x^{-\alpha/2} J_\alpha(2\sqrt{x\rho})=
\rho^{\alpha/2} e^{-\rho}\sum_{k=0}^\infty\sum_{n=0}^\infty\frac{(\alpha-\beta)_n\, \rho^{n+k}}
{\Gamma(\alpha+1+n+k)\,n!} L_k^\beta(x).
\]
Using (\ref{Gammaz+n}) produces a Kummer's confluent hypergeometric function of the first kind as the coefficient of the expansion.
Using the definition of Laguerre functions (\ref{Laguerredefn}) to replace the confluent hypergeometric function completes the proof.
$\hfill\blacksquare$

Consider the generating function (Srivastava \& Manocha 
(1984) \cite[p.~209]{SriManocha})
\begin{equation}
e^{-x\rho}=\frac{1}{(1+\rho)^\alpha}\sum_{n=0}^\infty\rho^n L_n^{\alpha-n}(x),
\label{Lagueregeneratingfnthm3}
\end{equation}
for $\alpha\in\C,$ $\rho\in\left\{z\in\C:|z|<1\right\},$ $x>0$. 
Using the connection relation for Laguerre polynomials (\ref{LagConnect}) 
in the generating function (\ref{Lagueregeneratingfnthm3}),
yields a double sum. Reversing the order of the
summation and shifting the $n$-index by $k$ produces
\[
e^{-x\rho}=\frac{1}{(1+\rho)^\alpha}\sum_{k=0}^\infty\sum_{n=0}^\infty  \frac{(\alpha-n-k-\beta)_n}{n!}\rho^{n+k} L_k^\beta(x).\nonumber
\]
Using 
(\ref{binomialexp}),
(\ref{Gammaz+n}), 
and 
substituting $z=\rho/(1+\rho)$ yields the known generating function for 
Laguerre polynomials Olver {\it et al.}~(2010) \cite[(18.12.13)]{NIST}, namely
\begin{equation}
\exp\left(\frac{x\rho}{\rho-1}\right)=(1-\rho)^{\beta+1}\sum_{n=0}^\infty\rho^n L_n^\beta(x),
\label{BLagueregeneratingfnthm3}
\end{equation}
for $\beta\in\C$.
Note that using the connection relation for Laguerre polynomials (\ref{LagConnect})
on (\ref{BLagueregeneratingfnthm3})
leaves this generating function invariant.

\begin{thm}
Let $\lambda\in\C$, $\alpha\in\C\setminus-\N,$ $\beta>-1,$ $\rho\in\left\{z\in\C:|z|<1\right\},$ $x>0$. Then
\[
M\left(\lambda,\alpha+1,\frac{x\rho}{\rho-1}\right)=
(1-\rho)^{\lambda}
\sum_{k=0}^\infty  \frac{(\lambda)_k}{(\alpha+1)_k}~{_2}F_1\left(
\begin{array}{c}
\lambda+k,\alpha-\beta\\[0.2cm]
\alpha+1+k
\end{array};\rho\right)
\rho^k L_k^\beta(x).\nonumber
\]
\end{thm}
\noindent {\bf Proof.}
On p.~132 of Srivastava \& Manocha (1984) \cite{SriManocha} there is a generating function
for Laguerre polynomials, namely
\[
M\left(\lambda,\alpha+1,\frac{x\rho}{\rho-1}\right)=
(1-\rho)^{\lambda}
\sum_{n=0}^\infty \frac{(\lambda)_n\,\rho^n}{(\alpha+1)_n}L_n^\alpha(x).\nonumber
\]
Using the connection relation for Laguerre polynomials (\ref{LagConnect}) 
we obtain a double summation. 
In order to justify reversing the resulting order 
of summation, we demonstrate 
(\ref{doublesumLaguerre}),
where
\[ c_n=\frac{(\lambda)_n\rho^n}{(\alpha+1)_n}\]
and $a_{nk}$ is given in (\ref{ank}).
We assume that $\alpha\in\C\setminus-\N$, $\beta>-1,$
$|\rho|<1$ and $x>0$.
Given (\ref{Laguerrebound}), (\ref{ankbound}), then
\[
\sum_{n=0}^\infty |c_n|\sum_{k=0}^n |a_{nk}| \left|L_k^\beta(x)\right|
\le K_1\sum_{n=0}^\infty \frac{|(\lambda)_n||\rho|^n}{|(\alpha+1)_n|} 
(1+n)^{\sigma_1+\sigma_2+1}\le K_3\sum_{n=0}^\infty|\rho|^n
(1+n)^{\sigma_1+\sigma_2+\lambda-\alpha} <\infty,
\]
for some $K_3\in\R$.  Reversing the order of the summation and shifting the $n$-index by $k$ produces
\[
M\left(\lambda,\alpha+1,\frac{x\rho}{\rho-1}\right)=
(1-\rho)^{\lambda}
\sum_{k=0}^\infty\sum_{n=0}^\infty \frac{(\lambda)_{n+k}\,(\alpha-\beta)_n}{(\alpha+1)_{n+k}\,n!}
\rho^{n+k} L_k^\beta(x).\nonumber
\]
Then, using (\ref{Gammaz+n}) with simplification completes the proof.
$\hfill\blacksquare$

\section{Expansions in Wilson polynomials}
\label{Wilsonpolynomials}

The Wilson 
polynomials $W_n\left(x^2;a,b,c,d\right),$ originally introduced in 
Wilson (1980) \cite{Wilson},
can be defined in terms of a terminating generalized
hypergeometric series as follows (Olver {\it et al.}~(2010) \cite[(18.26.1)]{NIST})\\
\[
W_n(x^2;a,b,c,d):=(a+b)_n(a+c)_n(a+d)_n\,{}_4F_3
\left(
\begin{array}{c}
-n,n+a+b+c+d-1,a+ix,a-ix\\[0.2cm]
a+b,a+c,a+d\end{array};1\right).
\]
These polynomials are perhaps the most general hypergeometric orthogonal polynomials
in existence being at the very top of the Askey scheme which classifies these
orthogonal polynomials (see for instance \cite[Figure 18.21.1]{NIST}).  
The orthogonality relation for Wilson polynomials can be found in
Koekoek {\it et al.}~(2010) \cite[Section 9.1]{Koekoeketal}, namely
\begin{eqnarray}
&&\hspace{-0.3cm}\int_0^\infty
\left|\frac{\Gamma(a+ix)\Gamma(b+ix)\Gamma(c+ix)\Gamma(d+ix)}{\Gamma(2ix)}\right|^2
W_m\left(x^2;a,b,c,d\right)W_n\left(x^2;a,b,c,d\right)dx\nonumber\\[0.2cm]
&&\hspace{0.7cm}=\frac{2\pi n!\,\Gamma(n+a+b)\Gamma(n+a+c)\Gamma(n+a+d)\Gamma(n+b+c)\Gamma(n+b+d)\Gamma(n+c+d)}
{(2n+a+b+c+d-1)\Gamma(n+a+b+c+d-1)}\delta_{m,n}\nonumber
\label{orthowilson}
\end{eqnarray}
where $
\Re\,a,
\Re\,b,
\Re\,c,
\Re\,d>0,$ and non-real parameters occurring in conjugate pairs.
A connection relation with one free parameter for the Wilson polynomials 
is given by 
\cite[equation just below (15)]{SanchezDehesa}, namely
\begin{eqnarray}
&&\hspace{-0.9cm}W_n\left(x^2;a,b,c,d\right)=\sum_{k=0}^n \frac{n!}{k!(n-k)!}
\,W_k\left(x^2;a,b,c,h\right)
\nonumber\\[0.2cm]
&&\hspace{0.3cm}\times \frac{(n+a+b+c+d-1)_k\,(d-h)_{n-k}\,(k+a+b)_{n-k}\,(k+a+c)_{n-k}\,(k+b+c)_{n-k}}
{(k+a+b+c+h-1)_k\,(2k+a+b+c+h)_{n-k}}.
\label{Wilsonconnect}
\end{eqnarray}

In this section, we give a generalization of a generating function for 
Wilson polynomials.
This example is intended to be illustrative.  In Koekoek, Lesky \& Swarttouw {\it et al.}~(2010) 
\cite{Koekoeketal} for instance there are
four separate generating functions given for the Wilson polynomials.
The technique applied in the proof of the theorem presented in this section, can 
be easily applied to the rest of the generating functions for Wilson polynomials
in Koekoek, Lesky \& Swarttouw {\it et al.}~(2010) \cite{Koekoeketal}.  Generalizations of these
generating functions (and their corresponding definite integrals) can be extended 
by a well-established limiting procedure (see \cite[Chapter 9]{Koekoeketal}) to 
the continuous dual Hahn, continuous Hahn, Meixner--Pollaczek, pseudo Jacobi, 
Jacobi, Laguerre and Hermite polynomials.

\begin{thm}
Let $\rho\in\left\{z\in\C:|z|<1\right\},$ $x\in(0,\infty)$,
$
\Re\,a,
\Re\,b,
\Re\,c,
\Re\,d,
\Re\,h>0$ 
and non-real parameters $a,b,c,d,h$ occurring in conjugate pairs.
Then
\begin{eqnarray}
&&\hspace{-0.95cm}{_2}F_1\left(\begin{array}{c}
a+ix,\,b+ix\\[0.1cm]
a+b\end{array};\rho\right)
{_2}F_1\left(\begin{array}{c}
c-ix,\,d-ix\\[0.1cm]
c+d\end{array};\rho\right)
\nonumber\\[0.2cm]
&&\hspace{-0.5cm}=\sum_{k=0}^\infty
\frac{(k+a+b+c+d-1)_{k}}
{(k+a+b+c+h-1)_{k}\,(a+b)_k\,(c+d)_k\,k!}\nonumber\\[0.2cm]
&&\hspace{-0.0cm}\times
\,{_4}F_3\left(\begin{array}{c}
d-h,\,2k+a+b+c+d-1,\,k+a+c,\,k+b+c\\[0.1cm]
k+a+b+c+d-1,\,2k+a+b+c+h,\,k+c+d\end{array};\rho\right)
\!\rho^k\,W_k\left(x^2;a,b,c,h\right).
\label{wilson1}
\end{eqnarray}
\end{thm}
\noindent {\bf Proof.} Koekoek {\it et al.}~(2010) \cite[(1.1.12)]{Koekoeketal} give a generating
function for Wilson polynomials, namely
\begin{eqnarray}
{_2}F_1\left(\begin{array}{c}
a+ix,\,b+ix\\[0.1cm]
a+b\end{array};\rho\right)
{_2}F_1\left(\begin{array}{c}
c-ix,\,d-ix\\[0.1cm]
c+d\end{array};\rho\right)
=\sum_{n=0}^\infty
\frac{\rho^n\,W_n\left(x^2;a,b,c,d\right)}{(a+b)_n\,(c+d)_n\,n!}.\nonumber
\end{eqnarray}
Using the connection relation for Wilson polynomials (\ref{Wilsonconnect}) in the above generating function
produces a double sum. 
In order to justify reversing the summation symbols we show that
\[
\sum_{n=0}^\infty |c_n|\sum_{k=0}^n |a_{nk}|\left|W_k(x^2;a,b,c,h)\right|< \infty ,
\]
where
\[ c_n=\frac{\rho^n}{(a+b)_n(c+d)_nn!} ,\]
and $a_{nk}$ are the connection coefficients satisfying
\[ W_n(x^2;a,b,c,d)=\sum_{k=0}^n a_{nk}W_k(x^2;a,b,c,h) .\]
We assume that $a,b,c,d$ and $a,b,c,h$ are positive except for complex conjugate pairs with positive real parts, and $x>0$.
It follows from \cite[bottom of page 59]{Wilson91} that
\begin{equation}\label{W2ineq1}
 \left|W_n(x^2;a,b,c,d)\right|\le K_1 (n!)^3 (1+n)^{\sigma_1},
\end{equation}
where $K_1$ and $\sigma_1$ are positive constants independent of $n$.

\begin{lemma}
Let $j\in\N$, $k,n\in\N_0$, 
$z\in \C$,
$\Re u>0$, 
$w>-1$, 
$v\ge 0$, 
$x>0$. 
Then
\begin{eqnarray}
&&\hspace{-5.2cm}|(u)_j|\ge  (\Re u) (j-1)!, 
\label{W2l2} \\[0.1cm]
&&\hspace{-5.2cm}\frac{(v)_n}{n!}\le (1+n)^v, 
\label{W2l3} \\[0.1cm]
&&\hspace{-5.2cm}(n+w)_k\le \max\{1,2^w\}\frac{(n+k)!}{n!},\qquad (k\le n), 
\label{W2l4} \\[0.1cm]
&&\hspace{-5.2cm}|(k+z)_{n-k}|\le (1+n)^{|z|} \frac{n!}{k!},\qquad (k\le n),  
\label{W2l5} \\[0.1cm]
&&\hspace{-5.2cm}(k+x-1)_k \ge \min\left\{\frac{x}{2},\frac16\right\} 
\frac{(2k)!}{k!}, 
\label{W2l6} \\[0.1cm]
&&\hspace{-5.2cm}(2k+x)_{n-k}\ge \min\{x,1\} \frac1{1+n}\frac{(n+k)!}{(2k)!},
\qquad (k\le n).  \label{W2l7}
\end{eqnarray}
\end{lemma}
\begin{proof}
Let us consider
\[
|(u)_j|= |u||u+1|\dots |u+j-1|
\ge \Re u (\Re u+1)\dots(\Re u+j-1)
\ge (\Re u)(j-1)!.
\]
This completes the proof of (\ref{W2l2}).
Choose $m\in\N_0$ such that $m\le v\le m+1$.
Then
\begin{eqnarray*}
\hspace{-2.0cm}\frac{(v)_n}{n!}&\le & \frac{(m+1)_n}{n!} 
=\frac{(n+1)(n+2)\dots(n+m)}{m!}
=\left(1+n\right) \left(1+\frac{n}{2}\right)\dots \left(1+\frac{n}{m}\right)\\
&\le & (1+n)^m\le (1+n)^v .
\end{eqnarray*}
This completes the proof of (\ref{W2l3}).
If $-1<w\le 1$ then
\[ n!(n+w)_k \le n!(n+1)(n+2)\dots(n+k)=(n+k)! .\]
If $m\le w\le m+1$ with $m\in\N$ then
\[
 n!(n+w)_k \le(n+k)!\frac{n+k+1}{n+1}\frac{n+k+2}{n+2}\dots\frac{n+m+k}{n+m}
 \le 2^m (n+k)!\le 2^w (n+k)!
\]
This completes the proof of (\ref{W2l4}). 
If $|z|\le 1$ then
\[ |(k+z)_{n-k}|\le (k+1)(k+2)\dots n =\frac{n!}{k!} .\]
If $|z|>1$ then, using (\ref{W2l3}),
\[
k! |(k+z)_{n-k}|\le  |z|(|z|+1)\dots(|z|+n-1)=(|z|)_n \le n!(1+n)^{|z|} .
\]
This completes the proof of (\ref{W2l5}).
Let $k\ge 2$. Then
\[
(k+x-1)_k\ge  (k-1)k\dots (2k-2) =\frac{k(k-1)}{2k(2k-1)} \frac{(2k)!}{k!} \ge \frac16 \frac{(2k)!}{k!} .
\]
The cases $k=0,1$ can be verified directly.
This completes the proof of (\ref{W2l6}).
Let $k\ge 1$. Then
\[
(2k+x)_{n-k} \ge (2k)_{n-k}
= \frac{(n+k)!}{(2k)!} \frac{2k}{n+k}\ge \frac{1}{1+n}\frac{(n+k)!}{k!} .
\]
The case $k=0$ can be verified separately.
This completes the proof of (\ref{W2l7}).
\end{proof}

Using (\ref{W2l2}), we obtain, for $n\in\N$,
\[
 |c_n|=  \frac{|\rho|^n}{|(a+b)_n||(c+d)_n| n!} 
 \le   \frac{|\rho^n|}{\Re(a+b)\Re(c+d)(n-1)!^2 n!}
 = \frac{1}{\Re(a+b)\Re(c+d)} \frac{n^2|\rho|^n}{(n!)^3} .
\]
 Therefore, we obtain, for all $n\in\N_0$,
\begin{equation}\label{W2ineq2}
|c_n|\le K_2 (1+n)^2 \frac{|\rho|^n}{(n!)^3} ,
\end{equation}
where
\[ K_2=\max\left\{1,\frac{1}{\Re(a+b)\Re(c+d)}\right\} .\]
Using (\ref{W2l3}), we find
\begin{equation}\label{W2a1}
\left|\frac{(d-h)_{n-k}}{(n-k)!}\right|\le  (1+n)^{|d-h|}.
\end{equation}
Using (\ref{W2l4}), we obtain
\begin{equation}\label{W2a2}
|n!(n+a+b+c+d-1)_k|\le  K_3(n+k)!.
\end{equation}
From (\ref{W2l5}), we find
\begin{equation}\label{W2a3}
|(k+a+b)_{n-k}|\le  (1+n)^{\sigma_4} \frac{n!}{k!},
\end{equation}
and similar estimates with $a+c$ and $b+c$ in place of $a+b$.
Using (\ref{W2l6}), we obtain
\begin{equation}\label{W2a4}
|(k+a+b+c+h-1)_k|\ge K_4 \frac{(2k)!}{k!},
\end{equation}
where $K_4>0$.
Using (\ref{W2l7}), we obtain
\begin{equation}\label{W2a5}
|(2k+a+b+c+h)_{n-k}|\ge  \frac{K_5}{1+n}\frac{(n+k)!}{(2k)!},
\end{equation}
where $K_5>0$.

Combining \eqref{W2a1}, \eqref{W2a2}, \eqref{W2a3}, \eqref{W2a4}, \eqref{W2a5}, we find
\begin{equation}\label{W2ineq3}
|a_{nk}|\le K_6 (1+n)^{\sigma_6} \left(\frac{n!}{k!}\right)^3 .
\end{equation}
Now \eqref{W2ineq1}, \eqref{W2ineq2}, \eqref{W2ineq3} give
\begin{eqnarray*}
 \sum_{n=0}^\infty |c_n|\sum_{k=0}^n |a_{nk}|\left|W_k(x^2;a,b,c,h)\right| &\le &
 K_1K_2K_6\sum_{n=0}^\infty |\rho|^n (1+n)^2 \sum_{k=0}^n (1+n)^{\sigma_1+\sigma_6}\\
&= & K_1K_2K_6\sum_{n=0}^\infty |\rho|^n (1+n)^{\sigma_1+\sigma_6+3} <\infty
\end{eqnarray*}
since $|\rho|<1$.
Reversing the order of the summation and shifting the $n$-index by $k$ produces
the generalized expansion (\ref{wilson1}).
$\hfill\blacksquare$

\appendix
\section{Definite integrals}

As a consequence of the series expansions given above, one may
generate corresponding definite integrals (in a one-step procedure)
as an application of the orthogonality relation for these 
hypergeometric orthogonal polynomials.  Integrals of such 
sort are always of interest since they are very likely
to find applications in applied mathematics and theoretical physics.

\begin{cor}
Let $k\in\N_0$, 
$\alpha\in\C$, $\beta,\gamma>-1$ such that if $\beta,\gamma\in(-1,0)$ then $\beta+\gamma+1\neq 0$,
$\rho\in\{z\in\C:|z|<1\}$. Then
\begin{eqnarray}
&&\hspace{-0.0cm}\int_{-1}^1 \frac{(1-x)^\gamma(1+x)^\beta}
{\RR\left(1+\RR-\rho\right)^\alpha\left(1+\RR+\rho\right)^\beta} P_k^{(\gamma,\beta)}(x)dx\nonumber\\[0.2cm]
&&\hspace{+3.1cm}=\frac{2^{1+\gamma-\alpha}\Gamma(\gamma+k+1)\Gamma(\beta+k+1)
\left(\frac{\alpha+\beta+1}{2}\right)_k\,\left(\frac{\alpha+\beta+2}{2}\right)_k}
{\Gamma(\gamma+\beta+2)(\alpha+\beta+1)_k\,
\left(\frac{\gamma+\beta+2}{2}\right)_k\,\left(\frac{\gamma+\beta+3}{2}\right)_k\,k!}\nonumber\\[0.2cm]
&&\hspace{+6.7cm}\times\,{_3}F_2\left(\begin{array}{c}
\beta+k+1, \alpha+\beta+2k+1, \alpha-\gamma\\[0.1cm]
\alpha+\beta+k+1, \gamma+\beta+2k+2\end{array};\rho\right)\rho^k.\nonumber
\end{eqnarray}
\label{jacgenint1}
\end{cor}
\noindent {\bf Proof.}
Multiplying both sides of (\ref{JacGenGen}) by $P_n^{(\gamma,\beta)}(x)(1-x)^\gamma(1+x)^\beta$
and integrating from $-1$ to $1$ using the orthogonality relation for Jacobi 
polynomials (\ref{JacobiOrthogonality}) with simplification completes the proof.
$\hfill\blacksquare$

\begin{cor}
Let $k\in\N_0$, 
$\alpha\in\C$, $\beta,\gamma>-1$ such that if $\beta,\gamma\in(-1,0)$ then $\beta+\gamma+1\neq 0$,
$\rho\in\{z\in\C:|z|<1\}$. Then
\begin{eqnarray}
&&\hspace{-0.5cm}\int_{-1}^1(1-x)^{\gamma-\alpha/2}(1+x)^{\beta/2}
J_\alpha\left(\sqrt{2(1-x)\rho}\right)I_\beta\left(\sqrt{2(1+x)\rho}\right)
P_k^{(\gamma,\beta)}(x)dx\nonumber\\[0.2cm]
&&\hspace{+1.0cm}=\frac{2^{\gamma+\beta/2-\alpha/2+1}\,
\Gamma(\gamma+k+1)\left(\frac{\alpha+\beta+1}{2}\right)_k\,
\left(\frac{\alpha+\beta+2}{2}\right)_k}
{\Gamma(\gamma+\beta+2)\,\Gamma(\alpha+k+1)\,(\alpha+\beta+1)_k\,
\left(\frac{\gamma+\beta+2}{2}\right)_k\,\left(\frac{\gamma+\beta+3}{2}\right)_k\,k!}\nonumber\\[0.2cm]
&&\hspace{+3.0cm}\times\,{_2}F_3\left(\begin{array}{c}
2k+\alpha+\beta+1,\alpha-\gamma\\[0.1cm]
\alpha+\beta+k+1,\gamma+\beta+2k+2,\alpha+1+k\end{array};\rho\right)\rho^{\alpha/2+\beta/2+k}.\nonumber
\end{eqnarray}
\end{cor}
\noindent {\bf Proof.}
Same as the proof of Corollary \ref{jacgenint1}, except apply
to both sides of (\ref{Jac2F3}).
$\hfill\blacksquare$

\begin{cor}
Let $\beta\in\C$, $\alpha,\gamma>-1$ such that if $\alpha,\gamma\in(-1,0)$ then $\alpha+\gamma+1\neq 0$,
$\rho\in\{z\in\C:|z|<1\}\setminus(-1,0]$. Then
\begin{eqnarray}
&&\hspace{-0.5cm}\int_{-1}^1 \frac{(1+x)^{\beta/2}(1-x)^\gamma}
{\RR^{\alpha+1}}
P_\alpha^{-\beta}\left(\frac{1+\rho}{\RR}\right)P_k^{(\gamma,\beta)}(x)dx\nonumber\\[0.2cm]
&&\hspace{+4.0cm}=\frac{2^{\gamma+\beta/2+1}\Gamma(\gamma+k+1)(\alpha+\beta+1)_{2k}}
{(1-\rho)^{\alpha-\gamma}\rho^{(\gamma+1)/2}k!}
P_{\gamma-\alpha}^{-\gamma-\beta-2k-1}\left(\frac{1+\rho}{1-\rho}\right).\nonumber
\end{eqnarray}
\end{cor}
\noindent {\bf Proof.}
Same as the proof of Corollary \ref{jacgenint1}, except apply
to both sides of (\ref{Jacwithalpha}).
$\hfill\blacksquare$

\begin{cor}
Let $\beta\in\C$, $\alpha,\gamma>-1$ such that if $\alpha,\gamma\in(-1,0)$ then $\alpha+\gamma+1\neq 0$,
$\rho\in(0,1)$. Then
\begin{eqnarray}
&&\hspace{-0.5cm}\int_{-1}^1 \frac{(1-x)^{\alpha/2}(1+x)^\gamma}
{\RR^{\beta+1}}
{\mathrm P}_\beta^{-\alpha}\left(\frac{1-\rho}{\RR}\right)
P_k^{(\alpha,\gamma)}(x)dx\nonumber\\[0.2cm]
&&\hspace{+4.0cm}=\frac{2^{\gamma+\alpha/2+1}\Gamma(\gamma+k+1)(\alpha+\beta+1)_{2k}}
{(1+\rho)^{\beta-\gamma}\rho^{(\gamma+1)/2}k!}
P_{\gamma-\beta}^{-\gamma-\alpha-2k-1}\left(\frac{1-\rho}{1+\rho}\right).\nonumber
\end{eqnarray}
\end{cor}
\noindent {\bf Proof.}
Same as the proof of Corollary \ref{jacgenint1}, except apply
to both sides of (\ref{Jacwithalphacom}).
$\hfill\blacksquare$

\begin{cor}
Let $n\in\N_0,$ 
$\alpha,\mu\in\C$, $\nu\in(-1/2,\infty)\setminus\{0\}$,
$\rho\in(0,1)$. Then
\begin{eqnarray}
&&\hspace{-0.15cm}\int_{-1}^1\left(1-x^2\right)^{\nu-\mu/2-1/4}
P_{\mu-\alpha-1/2}^{1/2-\mu}\left(
\RR
+\rho\right)
{\mathrm P}_{\mu-\alpha-1/2}^{1/2-\mu}\left(
\RR
-\rho\right)
C_n^\nu(x)dx\nonumber\\[0.2cm]
&&\hspace{+2.5cm}=\frac{\sqrt{\pi}2^{1/2-\mu}(2\nu)_n(\alpha)_n(2\mu-\alpha)_n(\mu)_n\Gamma(\frac12+\nu)}
{(2\mu)_n\Gamma(\frac12+\mu+n)\Gamma(1+\nu+n)\Gamma(\frac12+\mu)n!}\rho^{n+\mu-1/2}\nonumber\\[0.2cm]
&&\hspace{+4.5cm}\times\,{_6}F_5\left(
\begin{array}{c}
\frac{\alpha+n}{2},\frac{\alpha+n+1}{2},\frac{2\mu-\alpha+n}{2},
\frac{2\mu-\alpha+n+1}{2}
,\mu+n
,\mu-\nu
\\[0.2cm]
\frac{2\mu+n}{2},\frac{2\mu+n+1}{2},\frac{\mu+n+\frac12}{2},\frac{\mu+n+\frac32}{2},1+\nu+n
\end{array};\rho^2\right).\nonumber
\end{eqnarray}
\label{posdefintGegen}
\end{cor}
\noindent {\bf Proof.}
Multiplying both sides of (\ref{Gegenwith6F5}) by $C_n^\nu(x)(1-x^2)^{\nu-1/2}$ and integrating
from $-1$ to $1$ using the orthogonality relation for Gegenbauer 
polynomials (\ref{OrthoGegen})
with simplification completes the proof.
$\hfill\blacksquare$

\begin{cor}
Let $k\in\N_0$, 
$\alpha\in\C$, $\gamma\in(-1/2,\infty)\setminus\{0\}$,
$\rho\in\{z\in\C:|z|<1\}$. Then
\begin{eqnarray}
&&\hspace{-0.0cm}\int_{-1}^1 \frac{(1-x^2)^{\gamma-1/2}}
{\RR\left(1+\RR-\rho\right)^{\alpha-1/2}\left(1+\RR+\rho\right)^{\gamma-1/2}}C_k^\gamma(x)dx\nonumber\\[0.2cm]
&&\hspace{+0.3cm}=\frac{\sqrt{\pi}\,2^{1-\gamma-\alpha}\Gamma(\gamma+1/2)
\left(\frac{\alpha+\gamma}{2}\right)_k\,\left(\frac{\alpha+\gamma+1}{2}\right)_k\,(2\gamma)_k}
{\Gamma(\gamma+k+1)(\alpha+\gamma)_k\,k!}\,
{_3}F_2\left(\begin{array}{c}
\gamma+k+\frac12,\alpha+\gamma+2k,\alpha-\gamma\\[0.1cm]
\alpha+\gamma+k,2\gamma+2k+1\end{array};\rho\right)\rho^k.\nonumber
\end{eqnarray}
\end{cor}
\noindent {\bf Proof.}
Same as in the proof of Corollary \ref{posdefintGegen}, except
apply to both sides of (\ref{GegenFromJacGen}).
$\hfill\blacksquare$

\begin{cor}
Let $k\in\N_0$, $\alpha\in\C$,
$\rho\in\{z\in\C:|z|<1\}$. Then
\begin{eqnarray}
\int_{-1}^1\frac{(1+\RR+\rho)^{1/2}}{\RR(1+\RR-\rho)^{\alpha-1/2}(1-x^2)^{1/2}}T_k(x)dx
=\frac{\pi\left(\frac\alpha2\right)_k\,\left(\frac{\alpha+1}{2}\right)_k}
{2^{\alpha-1}(\alpha)_k\,k!}\,
{_3}F_2\left(\begin{array}{c}
k+\frac12,\alpha+2k,\alpha\\[0.1cm]
2k+1,\alpha+k\end{array};\rho\right)\rho^k.\nonumber
\end{eqnarray}
\end{cor}
\noindent {\bf Proof.}
Multiplying both sides of (\ref{ChebyT3F2}) by $T_k(x)(1-x^2)^{-1/2}$ and integrating from $-1$ to $1$
using the orthogonality relation for Chebyshev polynomials 
of the first kind,
Olver {\it et al.}~(2010) \cite[(18.2.1), (18.2.5), Table 18.3.1]{NIST}
\[
\int_{-1}^1
T_m(x)
T_n(x)
(1-x^2)^{-1/2}
dx
=
\frac{\pi}{\epsilon_n}
\delta_{m,n},
\]
with simplification completes the proof.
$\hfill\blacksquare$

\begin{cor}
Let $k\in\N_0$, $\alpha,\beta\in\R$, $\rho\in\C\setminus\{0\}$. Then
\[
\int_0^\infty x^{\beta-\alpha/2} e^{-x}  J_\alpha(2\sqrt{\rho x})L_k^\beta(x)dx
=\Gamma(\beta-\alpha+1) \frac{e^{-\rho}\rho^{k+\alpha/2}}{k!} L_{\beta-\alpha}^{\alpha+k}(\rho).\nonumber
\]
\label{LagInt1}
\end{cor}
\noindent {\bf Proof.}
Multiplying both sides of (\ref{LagforInt1}) by $x^\beta e^{-x} L_{k'}^\beta(x)$ for $k'\in\N_0$, integrating over $(0,\infty)$
and using the orthogonality relation for Laguerre polynomials (\ref{OrthoLag}) completes the proof.
$\hfill\blacksquare$

\medskip

\noindent Applying the process of Corollary \ref{LagInt1} to both 
sides of (\ref{BLagueregeneratingfnthm3})
produces the definite integral for Laguerre polynomials
\begin{equation}
\int_0^\infty x^\beta \exp\left(\frac{x}{\rho-1}\right) L_n^\beta(x)dx
=\frac{\Gamma(n+\beta+1)(1-\rho)^{\beta+1}}{n!}\rho^n,
\label{LagInt2}
\end{equation}
which is a specific case of the definite integral given by Gradshteyn \& Ryzhik (2007)
\cite[(7.414.8)]{Grad}. This is not surprising since (\ref{LagInt2}) was found using the
generating function for Laguerre polynomials.
$\hfill\blacksquare$

\begin{cor}
Let $k\in\N_0$, $\rho\in\left\{z\in\C:|z|<1\right\},$ 
$\Re\,a,
\Re\,b,
\Re\,c,
\Re\,d,
\Re\,h>0$ and non-real parameters occurring in conjugate pairs.
Then
\begin{eqnarray*}
&&\hspace{-0.10cm}\int_0^{\infty} 
\,{_2}F_1\left(\begin{array}{c}
a+ix,\,b+ix\\[0.1cm]
a+b\end{array};\rho\right)
{_2}F_1\left(\begin{array}{c}
c-ix,\,d-ix\\[0.1cm]
c+d\end{array};\rho\right)
W_k\left(x^2;a,b,c,h\right)
w(x)\,dx\\[0.2cm]
&&\hspace{1.9cm}
=\frac
{
2\pi\Gamma(a+b)
\Gamma(k+a+c)
\Gamma(k+a+h)
\Gamma(k+b+c)
\Gamma(k+b+h)
\Gamma(k+c+h)
}
{(c+d)_k
\Gamma(2k+a+b+c+h)
\left\{(k+a+b+c+d-1)_{k}\right\}^{-1}
}\nonumber\\[0.2cm]
&&\hspace{4.4cm}\times\,{}_4F_3\left(
\begin{array}{c}
d-h,2k+a+b+c+d-1,k+a+c,k+b+c\\[0.2cm]
k+a+b+c+d-1,2k+a+b+c+h,k+c+d
\end{array};\rho
\right)\rho^k.
\end{eqnarray*}
where $w:(0,\infty)\to\R$ is defined by 
\[
w(x):=\left|\frac{\Gamma(a+ix)\Gamma(b+ix)\Gamma(c+ix)\Gamma(h+ix)}
{\Gamma(2ix)}\right|^2.
\]
\end{cor}
\noindent {\bf Proof.}
Multiplying both sides of 
(\ref{wilson1})
by 
\[
\left|
\frac{\Gamma(a+ix)\Gamma(b+ix)\Gamma(c+ix)\Gamma(h+ix)}
{\Gamma(2ix)}
\right|^2W_{k'}(x^2;a,b,c,h),
\]
for $k'\in\N_0$,
integrating over $x\in(0,\infty)$
and using the orthogonality relation for Wilson polynomials 
(\ref{orthowilson})
completes the proof.
$\hfill\blacksquare$

\subsection*{Acknowledgements}

This work was conducted while H.~S.~Cohl was a National 
Research Council Research Postdoctoral Associate in the 
Applied and Computational Mathematics Division
at the National Institute of Standards and Technology, 
Gaithersburg, Maryland, U.S.A.
C.~MacKenzie would like to thank the Summer Undergraduate 
Research Fellowship program at the National Institute of 
Standards and Technology for financial support while this 
research was carried out.


\end{document}